\documentclass[12pt]{amsart}

\title{Enumeration of planar Tangles}
\author{Douglas A. Torrance}
\email{dtorrance@piedmont.edu}
\address{Department of Mathematical Sciences \\ Piedmont College \\ PO
  Box 10 \\ 1021 Central Ave \\ Demorest, GA 30535}

\usepackage{hyperref}
\usepackage{diagbox}
\usepackage{physics}
\usepackage{subcaption}
\usepackage{tikz}

\usetikzlibrary{calc}

\usepackage{chngcntr}
\counterwithin{figure}{section}
\counterwithin{table}{section}

\usepackage{tikz}

\newtheorem{theorem}{Theorem}[section]
\newtheorem{lemma}[theorem]{Lemma}


\begin{document}

\maketitle

\begin{abstract}
  A planar Tangle is a smooth simple closed curve piecewise defined by quadrants of circles with constant curvature.  We can enumerate Tangles by counting their dual graphs, which consist of a certain family of polysticks.  The number of Tangles with a given length or area grows exponentially, and we show the existence of their growth constants by comparing Tangles to two families of polyominoes.
\end{abstract}

\section{Introduction}

Suppose a child plays with a toy train set.  Given a certain number of pieces of track, how many different railroad configurations may be created?  We may attempt to answer a simplified version of this question, in which all the pieces of track are quarter-circles, by enumerating a particular combinatorial object.

A \textit{Tangle} is a smooth simple closed curve which is piecewise defined by quadrants of circles, each with the same curvature.  These quadrants are the \textit{links} of a Tangle, and the \textit{joints} are the points at which the links intersect.  A Tangle which is also a plane curve is known as a \textit{planar Tangle} or \textit{Tanglegram}.  We will avoid use of the latter term, as to a biologist, a tanglegram is a diagram used to compare phylogenetic trees \cite{treemap}.  Furthermore, as we focus only on planar Tangles in this paper, we will say ``Tangle'' when we really mean ``planar Tangle.''

The number of links in a Tangle its \textit{length}.  A Tangle with length $n$ is referred to as an $n$-Tangle.  The length of a Tangle is necessarily a multiple of four \cite{fleron1}, and so we say that a $4c$-Tangle has \textit{class} $c$.

Much of the previous work on Tangles has dealt with the transformation of one $n$-Tangle to another using various reflections and translations.  In fact, a conjecture made in \cite{chan,fleron1} that any $n$-Tangle could be obtained in this way from any other $n$-Tangle was proven false in \cite{taylor}.  Tangles have also been used for Truchet tilings of the plane \cite{browne}, to create fonts \cite{taylor}, and to introduce mathematics to liberal arts students \cite{hotchkiss}.  \textit{Tangloids}, which are the roulettes generated by a circle rolling along the interior of a Tangle, have also been studied \cite{schumacher1,schumacher2}.

In order to enumerate Tangles, we must first decide when two Tangles are equivalent.  First, after scaling and rotating, we may assume that any two Tangles have the same curvature and that all of their joints lie on one of the intercardinal points (i.e., northeast, southeast, southwest, northwest) of the corresponding circles.

If we may translate one Tangle to obtain another, then they are equivalent as \textit{fixed} Tangles.  If we may translate and rotate one Tangle to obtain another, then they are equivalent as \textit{one-sided} Tangles.  If we may translate, rotate, and reflect one Tangle to obtain another, then they are equivalent as \textit{free} Tangles.

The shortest Tangles are the \textit{circle}, the only 4-Tangle of any variety, and the \textit{dumbbell}, the only free or one-sided 8-Tangle  (see Figure \ref{circle and dumbbell}).  The dumbbell corresponds to two fixed 8-Tangles.  The number of free $4c$-Tangles is known for $c\leq 11$ \cite[A000644]{oeis}.

\begin{figure}
  \begin{tikzpicture}
    \draw[line width=4pt] (0,0) circle (0.5);
    \draw[line width=4pt] (2,0) + (45:0.5) arc (45:315:0.5);
    \draw[line width=4pt] (2,0) + (45:0.5) arc (225:315:0.5);
    \draw[line width=4pt] (2,0) + (315:0.5) arc (135:45:0.5);
    \draw[line width=4pt] ({2 + sqrt(2)}, 0) + (135:0.5) arc (135:-135:0.5);
    \fill[white] (45:0.5) circle (1.5pt);
    \fill[white] (135:0.5) circle (1.5pt);
    \fill[white] (-45:0.5) circle (1.5pt);
    \fill[white] (-135:0.5) circle (1.5pt);
    \fill[white] (2,0) + (45:0.5) circle (1.5pt);
    \fill[white] (2,0) + (135:0.5) circle (1.5pt);
    \fill[white] (2,0) + (-45:0.5) circle (1.5pt);
    \fill[white] (2,0) + (-135:0.5) circle (1.5pt);
    \fill[white] ({2 + sqrt(2)},0) + (45:0.5) circle (1.5pt);
    \fill[white] ({2 + sqrt(2)},0) + (135:0.5) circle (1.5pt);
    \fill[white] ({2 + sqrt(2)},0) + (-45:0.5) circle (1.5pt);
    \fill[white] ({2 + sqrt(2)},0) + (-135:0.5) circle (1.5pt);
     \end{tikzpicture}
  \caption{circle (left) and dumbbell (right)}
  \label{circle and dumbbell}
\end{figure}
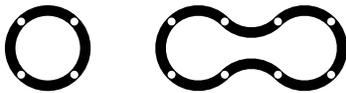

In \S 2, we introduce the \textit{dual graph} of a Tangle, a combinatorial description first appearing in \cite{taylor}.  We then use the dual graph to compute the area enclosed by a Tangle (\S 3) and to develop an algorithm for the enumeration of Tangles (\S 4).  In \S 5, we introduce two additional combinatorial descriptions of Tangles involving polyominoes suggested by \cite{chan,fleron1}, which we use in \S 6 to study the growth constants of Tangles relative to both area and length.

\section{Dual Graphs}

Consider the circles in a square packing of the plane.
We can color the circles with two colors, say black and white like a chessboard, so that no two kissing circles  have the same color.

The links of a Tangle all belong to circles in such a packing.
As a particle moves along a Tangle, it will occasionally (except for the case of the trivial 4-Tangle) move from one circle to another.  When it does this, it will switch from a link belonging to a circle lying inside the Tangle to a link belonging to a circle lying outside the Tangle, or vice versa.
These interior and exterior circles will have opposite colors, say black and white, respectively.
See Figure \ref{circle packing}.

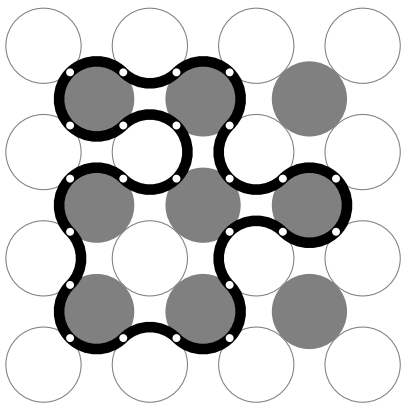
\begin{figure}
  \begin{tikzpicture}
    \foreach \x in {1,2,3,4}
    \foreach \y in {1,2,3,4}
      \draw[gray] ({\x*sqrt(2)}, {\y*sqrt(2)}) circle (0.5);
    \foreach \x in {1,2,3}
    \foreach \y in {1,2,3} {
      \fill[gray] ({(\x+1/2)*sqrt(2)}, {(\y + 1/2)*sqrt(2)}) circle (0.5);
 (1.5pt);
    }
    \draw[line width=4pt] ({sqrt(2)},{2*sqrt(2)}) + (45:0.5) arc (225:45:0.5)
    arc (225:315:0.5) arc (-45:135:0.5) arc (315:45:0.5) arc (225:315:0.5)
    arc (135:-45:0.5) arc (135:315:0.5) arc (135:-135:0.5) arc (45:135:0.5)
    arc (135:225:0.5) arc (45:-135:0.5) arc (45:135:0.5) arc (315:135:0.5)
    arc (-45:45:0.5);
    \foreach \x in {1, 2}
    \foreach \i in {0, 1, 2, 3} {
      \fill[white] ({(\x + 1/2)*sqrt(2)}, {3.5 * sqrt(2)}) + ({45 + 90 * \i}:0.5) circle (1.5pt);
    }
    \foreach \i in {0, 1, 2, 3} {
      \fill[white] ({3.5 * sqrt(2)}, {2.5 * sqrt(2)}) + ({45 + 90 * \i}:0.5) circle (1.5pt);
    }
    \foreach \i in {0, 1, 2} {
      \fill[white] ({1.5 * sqrt(2)}, {2.5 * sqrt(2)}) + ({45 + 90 * \i}:0.5) circle (1.5pt);
    }
    \foreach \i in {0, 1, 3} {
      \fill[white] ({2.5 * sqrt(2)}, {2.5 * sqrt(2)}) + ({45 + 90 * \i}:0.5) circle (1.5pt);
    }
    \foreach \i in {1, 2, 3} {
      \fill[white] ({1.5 * sqrt(2)}, {1.5 * sqrt(2)}) + ({45 + 90 * \i}:0.5) circle (1.5pt);
    }
    \foreach \i in {0, 2, 3} {
      \fill[white] ({2.5 * sqrt(2)}, {1.5 * sqrt(2)}) + ({45 + 90 * \i}:0.5) circle (1.5pt);
    }
  \end{tikzpicture}
  \caption{24-Tangle with circles in a square packing of the plane}
  \label{circle packing}
\end{figure}

We define the \textit{dual graph} of a Tangle as follows.  At the center of every black circle inside the Tangle, draw a vertex.  Connect two vertices with an edge if it does not intersect the Tangle and intersects no other circles in the packing.  Note that every vertex is incident to at most four edges, one for each of the cardinal directions.  See Figure \ref{dual graph example}.

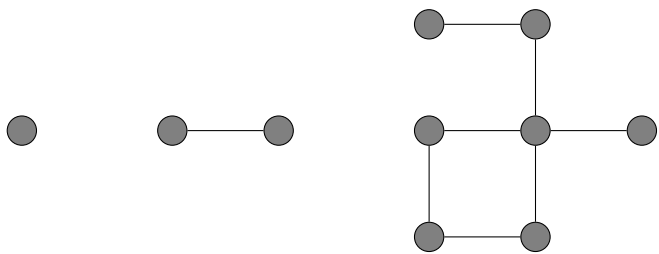
\begin{figure}
  \begin{tikzpicture}
    \path (0,0) node[draw,shape=circle,fill=gray] (a) {};
    \path ({sqrt(2)},0) node[draw,shape=circle,fill=gray] (b) {};
    \path (0,{sqrt(2)}) node[draw,shape=circle,fill=gray] (c) {};
    \path ({sqrt(2)},{sqrt(2)}) node[draw,shape=circle,fill=gray] (d) {};
    \path ({2*sqrt(2)},{sqrt(2)}) node[draw,shape=circle,fill=gray] (e) {};
    \path (0,{2*sqrt(2)}) node[draw,shape=circle,fill=gray] (f) {};
    \path ({sqrt(2)},{2*sqrt(2)}) node[draw,shape=circle,fill=gray] (g) {};
    \path (-2, {sqrt(2)}) node[draw,shape=circle,fill=gray] (h) {};
    \path ({-2 - sqrt(2)},{sqrt(2)}) node[draw,shape=circle,fill=gray] (i) {};
    \path ({-4 - sqrt(2)},{sqrt(2)}) node[draw,shape=circle,fill=gray] (j) {};
    \draw (d) -- (c) -- (a) -- (b) -- (d) -- (e);
    \draw (d) -- (g) -- (f);
    \draw (h) -- (i);
  \end{tikzpicture}
  \caption{Dual graphs of the circle, dumbbell, and the 24-Tangle in Figure \ref{circle packing}}
  \label{dual graph example}
\end{figure}

The dual graph of a Tangle is not a graph in the strictest sense of the word, as the positions of the vertices and edges in relation to one another have meaning.
Instead it is a \textit{polystick} \cite{barwell}.
However, we will still use familiar graph theoretical terminology to describe a dual graph.  For example, if the dual graph of a Tangle contains no cycles, then we say that it is a tree.

Not every polystick is the dual graph of some Tangle.  Indeed, a cycle in a dual graph must be ``filled in'' with squares, as it corresponds to a cluster of black circles inside the Tangle.  In other words, the only chordless cycles in a dual graph are squares.

\begin{lemma}\label{class without squares}
  If a Tangle has a dual graph which is a tree with $m$ edges, then it has class $m+1$.
\end{lemma}

\begin{proof}
  We use induction on $m$.  For the base case, the only Tangle whose dual graph has no edges is the circle, which has class 1.

  Suppose $m\geq 1$ and we remove from the dual graph an edge which is incident to a leaf, i.e., a degree 1 vertex.  Then we obtain a dual graph with $m-1$ edges, which by induction corresponds to a Tangle with class $m$.  But in doing this, we remove 5 links and add 1, a net change of $-4$ (see Figure \ref{class without squares proof}), and so the original Tangle has class $m+1$.
\end{proof}

\begin{figure}
  \begin{tikzpicture}
    \path (0,0) node[draw,shape=circle,fill=gray] (a) {};
    \path ({sqrt(2)},0) node[draw,shape=circle,fill=gray] (b) {};
    \draw[line width=4pt] (-45:0.5) arc (-45:45:0.5);
    \draw[line width=4pt,color=lightgray] (45:0.5) arc (225:315:0.5);
    \draw[line width=4pt,color=lightgray] (b) + (-135:0.5) arc (-135:135:0.5);
    \draw[line width=4pt,color=lightgray] (-45:0.5) arc (135:45:0.5);
    \draw[line width=4pt] (-45:0.5) arc (-45:45:0.5);
    \draw[dashed] (a) -- (b);
    \foreach \i in {0, 3} {
      \fill[white] ({45 + 90 * \i}:0.5) circle (1.5pt);
    }
    \foreach \i in {0,1, 2, 3} {
      \fill[white] ({sqrt(2)}, 0) + ({45 + 90 * \i}:0.5) circle (1.5pt);
    }
  \end{tikzpicture}
  \caption{Proof of Lemma \ref{class without squares}}
  \label{class without squares proof}
\end{figure}

\begin{lemma}\label{class with squares}
  If a Tangle has a dual graph with $m$ edges and $k$ squares, then it has class $m-2k+1$.
\end{lemma}

\begin{proof}
  We use induction on $k$.  The base case, with no squares, reduces to Lemma \ref{class without squares}.

  Suppose $k\geq 1$.  If we remove one edge which belongs to exactly one square of the dual graph, then we add 5 links and remove 1 to obtain the corresponding Tangle, a net change of 4 (see Figure \ref{class with squares proof}).  By induction, this Tangle has class $m-1-2(k-1)+1=m-2k+2$, and so the original Tangle has class $m-2k+1$.
\end{proof}

\begin{figure}
  \begin{tikzpicture}
    \path (0,0) node[draw,shape=circle,fill=gray] (a) {};
    \path ({sqrt(2)},0) node[draw,shape=circle,fill=gray] (b) {};
    \path ({sqrt(2)}, {sqrt(2)}) node[draw,shape=circle,fill=gray] (c) {};
    \path (0, {sqrt(2)}) node[draw,shape=circle,fill=gray] (d) {};
    \draw[line width=4pt,color=lightgray] (b) + (45:0.5) arc (225:135:0.5);
    \draw[line width=4pt] (b) + (45:0.5) arc (45:135:0.5);
    \draw[line width=4pt] (b) + (135:0.5) arc (315:45:0.5);
    \draw[line width=4pt] (c) + (225:0.5) arc (225:315:0.5);
    \draw (c) -- (d) -- (a) -- (b);
    \draw[dashed] (b) -- (c);
    \foreach \i in {0, 1, 2, 3} {
      \fill[white] ({1/sqrt(2)}, {1/sqrt(2)}) + ({45 + 90 * \i}:0.5) circle (1.5pt);
    }
    \foreach \i in {1, 2} {
      \fill[white] ({3/sqrt(2)}, {1/sqrt(2)}) + ({45 + 90 * \i}:0.5) circle (1.5pt);
    }
  \end{tikzpicture}
  \caption{Proof of Lemma \ref{class with squares}}
  \label{class with squares proof}
\end{figure}

We define the \textit{size} of a Tangle to be the size of its dual graph, i.e., the number of edges.

\begin{theorem}\label{edge bound}
  If a Tangle has class $c$, then its size is between $c-1$ and $\frac{c^2-1}{2}$ edges.
\end{theorem}

\begin{proof}
  A dual graph containing $k$ squares has at least $2k + 2\sqrt k$ edges \cite[A078633]{oeis}.  Or equivalently, a dual graph with $m$ edges has at most $\frac{m+1-\sqrt{2m+1}}{2}$ squares.  So by Lemma \ref{class with squares}, $c\geq\sqrt{2m+1}$, or $m\leq\frac{c^2-1}{2}$.

  On the other hand, if a dual graph with $m$ edges contains no squares, then $m=c-1$ by Lemma \ref{class without squares}.
\end{proof}

\section{Area}

We now compute the area enclosed by a Tangle, which depends only on its size.  Our result generalizes that of \cite{browne}, in which the only Tangles considered are those whose dual graphs are trees.

\begin{lemma}\label{euler formula}
  If a Tangle has a dual graph with $v$ vertices, $k$ squares, and $m$ edges, then $v+k-m=1$.
\end{lemma}

\begin{proof}
  The dual graph is a planar graph with $k+1$ faces (the $k$ squares plus the external face).  So by Euler's polyhedral formula, we have
  \begin{equation*}
    v+k+1-m=2.\qedhere
  \end{equation*}
\end{proof}

\begin{theorem}\label{tangle area}
  If a Tangle with curvature $\frac{1}{r}$ has size $m$, then it encloses an area of $(4m + \pi)r^2$.
\end{theorem}

\begin{proof}
  Recall that we constructed the dual graph of a Tangle by looking at black and white circles in a square circle packing of the plane.  The black circles inside the Tangle correspond to vertices of the dual graph, the white circles inside the Tangle correspond to squares in the dual graph, and the remaining ``in-between'' regions (which look like, but are not, astroids) correspond to edges.  Each ``in-between'' region has the same area as the figure obtained by removing an inscribed circle of radius $r$ from a square of side length $2r$, i.e., $(2r)^2 - \pi r^2 = (4-\pi)r^2$.

  Therefore, if the dual graph has $v$ vertices, $k$ squares, and  $m$ edges, then the area enclosed by the Tangle is
  \begin{align*}
    \pi r^2 v + \pi r^2 k + (4 - \pi) r^2 m &= (4m+(v+k-m)\pi)r^2\\
    &= (4m + \pi)r^2,
  \end{align*}
  by Lemma \ref{euler formula}.
\end{proof}

\section{Enumeration}

By Theorem \ref{edge bound}, we see that to enumerate all $4c$-Tangles, we may construct all possible dual graphs containing between $c-1$ and $\left\lfloor\frac{c^2-1}{2}\right\rfloor$ edges.  To do this, we construct all polysticks with the given number of sticks, and then remove any containing chordless cycles which are not squares.

The standard technique for constructing polysticks is due to Redelmeier \cite{redelmeier}.  Although his original method was specific to polyominoes, it is readily adapted to polysticks \cite{malkis}.  This gives us all fixed polysticks up to a given size.

We then check for chordless cycles which are not squares.  This may be accomplished by ensuring that all elements of a fundamental set of cycles are squares using, e.g., Paton's algorithm \cite{paton}.  After throwing out any polysticks not satisfying this property, we are left with exactly the dual graphs of Tangles.  For each dual graph, we compute the length of the corresponding Tangle using Lemma \ref{class with squares}.

Some of these may be equivalent under the eight transformations in the dihedral group of symmetries of the square.  We remove these redundancies to obtain the dual graphs of all the one-sided and free Tangles.

The author implemented the above process in Python \cite{python} using NetworkX \cite{networkx}, obtaining the results in Table \ref{enumeration results} \cite{tanglenum}.

Dual graphs which are trees have been well-studied, and are often known in the literature as \textit{bond trees} on the square lattice.  In particular, the main diagonals of Tables \ref{fixed enumeration} \cite{gstw} and \ref{free enumeration} \cite[A056841]{oeis} are already known.

\begin{table}[p]
  \begin{subtable}{\linewidth}
    \centering
    \begin{tabular}{|r||rrrrrrrrrrr|}
      \hline
      \backslashbox{$m$}{$c$} & 1 & 2 & 3 & 4 & 5 & 6 & 7 & 8 & 9 & 10 & 11 \\
      \hline\hline
      0 & 1 & & & & & & & & & & \\
      1 & & 2 & & & & & & & & & \\
      2 & & & 6 & & & & & & & & \\
      3 & & & & 22 & & & & & & & \\
      4 & & & 1 & & 87 & & & & & & \\
      5 & & & & 8 & & 364 & & & & & \\
      6 & & & & & 52 & & 1574 & & & & \\
      7 & & & & 2 & & 304 & & 6986 & & & \\
      8 & & & & & 22 & & 1706 & & 31581 & & \\
      9 & & & & & & 182 & & 9312 & & 144880 & \\
      10 & & & & & 6 & & 1288 & & 50056 & & 672390 \\
      \hline
    \end{tabular}
    \caption{Fixed Tangles of size $m$ and class $c$}
    \label{fixed enumeration}
  \end{subtable}
  \begin{subtable}{\linewidth}
    \centering
    \begin{tabular}{|r||rrrrrrrrrrr|}
      \hline
      \backslashbox{$m$}{$c$} & 1 & 2 & 3 & 4 & 5 & 6 & 7 & 8 & 9 & 10 & 11 \\
      \hline\hline
      0 & 1 & & & & & & & & & & \\
      1 & & 1 & & & & & & & & & \\
      2 & & & 2 & & & & & & & & \\
      3 & & & & 7 & & & & & & & \\
      4 & & & 1 & & 24 & & & & & & \\
      5 & & & & 2 & & 97 & & & & & \\
      6 & & & & & 14 & & 401 & & & & \\
      7 & & & & 1 & & 76 & & 1772 & & & \\
      8 & & & & & 6 & & 432 & & 7930 & & \\
      9 & & & & & & 49 & & 2328 & & 36335 & \\
      10 & & & & & 2 & & 326 & & 12534 & & 168249 \\
      \hline
    \end{tabular}
    \caption{One-sided Tangles of size $m$ and class $c$}
    \label{one-sided enumeration}
  \end{subtable}
  \begin{subtable}{\linewidth}
    \centering
    \begin{tabular}{|r||rrrrrrrrrrr|}
      \hline
      \backslashbox{$m$}{$c$} & 1 & 2 & 3 & 4 & 5 & 6 & 7 & 8 & 9 & 10 & 11 \\
      \hline\hline
      0 & 1 & & & & & & & & & & \\
      1 & & 1 & & & & & & & & & \\
      2 & & & 2 & & & & & & & & \\
      3 & & & & 5 & & & & & & & \\
      4 & & & 1 & & 15 & & & & & & \\
      5 & & & & 1 & & 54 & & & & & \\
      6 & & & & & 9 & & 212 & & & & \\
      7 & & & & 1 & & 38 & & 908 & & & \\
      8 & & & & & 4 & & 224 & & 4011 & & \\
      9 & & & & & & 28 & & 1164 & & 18260 & \\
      10 & & & & & 2 & & 170 & & 6299 & & 84320 \\
      \hline
    \end{tabular}
    \caption{Free Tangles of size $m$ and class $c$}
    \label{free enumeration}
  \end{subtable}
  \caption{Tangle enumeration results}
  \label{enumeration results}
\end{table}

\section{Polyominoes}

A \textit{polyomino} is a well-known combinatorial object consisting of a collection of squares of fixed area (called \textit{cells}) glued together along their edges.  We will see that there is a connection between Tangles and polyominoes.  As Tangles do not have holes, we restrict our attention to polyominoes without holes.  The boundaries of hole-free polyominoes are also known as \textit{self-avoiding polygons}.

Previous authors \cite{chan,fleron1} have drawn comparisons between Tangles and polyominoes.  We make their observations precise and find relationships between the length and area of the a Tangle and the perimeters and areas of the corresponding polyominoes.

The \textit{Chan polyomino} of a given Tangle is the polyomino obtained by adding a cell for every vertex of the dual graph so that if two vertices are adjacent, then the corresponding cells are adjacent.  Note that the converse is not true; if two cells in the Chan polyomino are adjacent, then the corresponding vertices in the dual graph may not be.  Furthermore, this is not a one-to-one correspondence, as different Tangles may have the same Chan polyomino.  However, every hole-free polyomino is the Chan polyomino of at least one Tangle.

The \textit{Fleron polyomino} of a given Tangle is the polyomino obtained by adding a cell for every vertex and edge of the dual graph so that if an edge is incident to a vertex, then the corresponding cells are adjacent.  Additional cells are then added to fill in all the holes corresponding to squares in the dual graph.  This is a one-to-one correspondence, but is not onto.  Indeed, a straight polyomino (in which all cells are arranged in a line) is only the Fleron polyomino of a Tangle if it contains an odd number of cells.

See Figure \ref{chan and fleron polyominoes} for an example of each of these polyominoes.

\begin{figure}
  \begin{tikzpicture}[scale=0.75]
    \coordinate (tangle_base) at ({0.5 + cos(135)}, {0.5 + sin(135)});
    \coordinate (chan_base) at ({2+2*sqrt(2)}, {sqrt(2)-1});
    \coordinate (fleron_base) at ($ (chan_base) + (4, -1) $);
    \draw[line width=4pt] (tangle_base) + (45:0.5)
    arc (225:45:0.5)
    arc (225:315:0.5) arc (-45:135:0.5) arc (315:45:0.5) arc (225:315:0.5)
    arc (135:-45:0.5) arc (135:315:0.5) arc (135:-135:0.5) arc (45:135:0.5)
    arc (135:225:0.5) arc (45:-135:0.5) arc (45:135:0.5) arc (315:135:0.5)
    arc (-45:45:0.5);
    \draw[ultra thick] (chan_base) -- ++ (2,0) -- ++ (0,1) -- ++ (1,0)
    -- ++ (0,1) -- ++ (-1,0) -- ++ (0,1) -- ++ (-2,0) -- ++ (0,-3);
    \draw[ultra thick] ($ (chan_base) + (1,0) $) -- ++ (0,3);
    \draw[ultra thick] ($ (chan_base) + (0,1) $) -- ++ (2,0) -- ++ (0,1)
    -- ++ (-2,0);
    \draw[ultra thick] (fleron_base) -- ++ (3,0) -- ++ (0,2) -- ++ (2,0)
    -- ++ (0,1) -- ++ (-2,0) -- ++ (0,2) -- ++ (-3,0) -- ++ (0,-1)
    -- ++ (2,0) -- ++ (0,-1) -- ++ (-2,0) -- ++ (0,-3);
    \draw[ultra thick] ($ (fleron_base) + (1,0) $) -- ++ (0,3);
    \draw[ultra thick] ($ (fleron_base) + (2,0) $) -- ++ (0,3);
    \draw[ultra thick] ($ (fleron_base) + (0,1) $) -- ++ (3,0);
    \draw[ultra thick] ($ (fleron_base) + (0,2) $) -- ++ (3,0);
    \draw[ultra thick] ($ (fleron_base) + (3,2) $) -- ++ (0,1) -- ++ (-1,0);
    \draw[ultra thick] ($ (fleron_base) + (3,4) $) -- ++ (-1,0) -- ++ (0,1);
    \draw[ultra thick] ($ (fleron_base) + (1,4) $) -- ++ (0,1);
    \draw[ultra thick] ($ (fleron_base) + (4,2) $) -- ++ (0,1);
    \foreach \x in {1, 2}
    \foreach \i in {0, 1, 2, 3} {
      \fill[white] ({0.5 + cos(135) + (2*\x - 1)/sqrt(2)}, {0.5 + sin(135) + 3/2*sqrt(2)}) + ({45 + 90 * \i}:0.5) circle (1.5pt);
    }
    \foreach \i in {0, 1, 2, 3} {
      \fill[white] ({0.5 + cos(135) + 5/sqrt(2)}, {0.5 + sin(135) + 1/2*sqrt(2)}) + ({45 + 90 * \i}:0.5) circle (1.5pt);
    }
    \foreach \i in {0, 1, 3} {
      \fill[white] ({0.5 + cos(135) + 3/sqrt(2)}, {0.5 + sin(135) + 1/2*sqrt(2)}) + ({45 + 90 * \i}:0.5) circle (1.5pt);
    }
    \foreach \i in {0, 1, 2} {
      \fill[white] ({0.5 + cos(135) + 1/sqrt(2)}, {0.5 + sin(135) + 1/2*sqrt(2)}) + ({45 + 90 * \i}:0.5) circle (1.5pt);
    }
    \foreach \i in {0, 2, 3} {
      \fill[white] ({0.5 + cos(135) + 3/sqrt(2)}, {0.5 + sin(135) - 1/2*sqrt(2)}) + ({45 + 90 * \i}:0.5) circle (1.5pt);
    }
    \foreach \i in {1, 2, 3} {
      \fill[white] ({0.5 + cos(135) + 1/sqrt(2)}, {0.5 + sin(135) - 1/2*sqrt(2)}) + ({45 + 90 * \i}:0.5) circle (1.5pt);
    }
  \end{tikzpicture}
  \caption{The 24-Tangle from Figure \ref{circle packing} with its Chan and Fleron polyominoes}
  \label{chan and fleron polyominoes}
\end{figure}
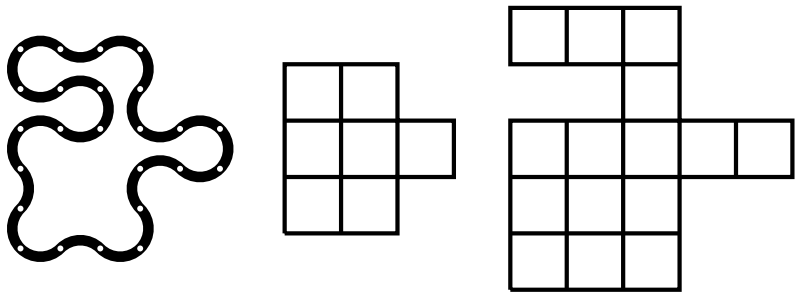

\begin{lemma}\label{chan area}
 Any hole-free polyomino with $m+1$ cells is the Chan polyomino of a Tangle with size $m$.
\end{lemma}

\begin{proof}
  First construct a dual graph by assigning to each cell of the polyomino a vertex and drawing edges when corresponding cells are adjacent.  Then delete edges to obtain a spanning tree.  This tree has $m$ edges, and by construction, it corresponds to a Tangle with the desired Chan polyomino.
\end{proof}

\begin{lemma}\label{fleron area}
 If a Tangle has size $m$, then its Fleron polyomino has $2m+1$ cells.
\end{lemma}

\begin{proof}
  We use induction on the number of squares in the dual graph.  If there are no squares, then the dual graph is a tree with $m$ edges and $m+1$ vertices.  Since there are no squares, there are then $2m+1$ cells in the Fleron polyomino.

  Now suppose we have a dual graph with $m$ edges and at least one square.  By induction, removing one edge belonging to exactly one square results in a dual graph corresponding to a polyomino with $2(m-1)+1=2m-1$ cells.  But by removing this one edge, we have removed two cells from the polyomino corresponding to the original dual graph -- one corresponding to the edge itself and one corresponding to the hole in the middle of the square.  Therefore, the polyomino had $2m+1$ cells.
\end{proof}

\begin{lemma}\label{chan perimeter}
  Any hole-free polyomino with perimeter $2(c+1)$ is the Chan polyomino of a Tangle of length $4c$.
\end{lemma}

\begin{proof}
  We use induction on the number of cells in the polyomino.  For the base case, a monomino, which has perimeter $4=2(1 + 1)$, is the Chan polyomino of a circle, which has length $4=4\cdot 1$.

  Now consider a hole-free polyomino with at least two cells and a perimeter of $2(c+1)$.  Construct the dual graph of a Tangle by adding a vertex for each cell and an edge where the corresponding cells are adjacent.

  If there exists a cell which is adjacent to only one other cell, then remove it.  The resulting polyomino has perimeter $2c$, since we removed three exterior edges but a formerly interior edge has been exposed to become an exterior edge.  When we remove the corresponding vertex and edge from the dual graph, we reduce the length of the Tangle by 4 as in Lemma \ref{class without squares}.  So by induction, the original Tangle has length $4(c-1)+4=4c$.

  If no such cell exists, then we remove a ``corner'' cell which is adjacent to exactly two other cells.  The resulting polyomino also has perimeter $2(c+1)$, as the two exterior edges which were removed are replaced by two formerly interior edges.  One vertex and two edges are removed in the corresponding dual graph, corresponding to the removal of four links which are offset by the addition of four more.  See Figure \ref{chan perimeter proof}.  This operation is known as an $\Omega$-rotation in the literature.  We can then continue to remove corner cells until there exists a cell adjacent to only one other cell and then proceed with our induction as above.
  \end{proof}

  \begin{figure}
  \begin{tikzpicture}
    \path (0,0) node[draw,shape=circle,fill=gray] (a) {};
    \path ({sqrt(2)},0) node[draw,shape=circle,fill=gray] (b) {};
    \path ({sqrt(2)}, {sqrt(2)}) node[draw,shape=circle,fill=gray] (c) {};
    \path (0, {sqrt(2)}) node[draw,shape=circle,fill=gray] (d) {};
    \draw[line width=4pt] (d) + (45:0.5) arc (45:-45:0.5) arc (135:315:0.5)
    arc (135:45:0.5);
    \draw[line width=4pt,lightgray] (d) + (45:0.5) arc (225:315:0.5) arc (135:-45:0.5)
    arc (135:215:0.5);
    \draw[line width=4pt] (d) + (45:0.5) arc (45:-45:0.5) arc (135:315:0.5)
    arc (135:45:0.5);
    \draw (d) -- (a) -- (b);
    \draw[dashed] (b) -- (c) -- (d);
    \fill[white] (a) + (45:0.5) circle (1.5pt);
    \foreach \i in {0, 1} {
      \fill[white] (b) + ({45 + 90 * \i}:0.5) circle (1.5pt);
    }
    \foreach \i in {0, 1, 3} {
      \fill[white] (c) + ({45 + 90 * \i}:0.5) circle (1.5pt);
    }
    \foreach \i in {0, 3} {
      \fill[white] (d) + ({45 + 90 * \i}:0.5) circle (1.5pt);
    }

  \end{tikzpicture}
  \caption{Proof of Lemma \ref{chan perimeter}}
  \label{chan perimeter proof}
\end{figure}

\begin{lemma}\label{fleron perimeter}
  If a Tangle has length $4c$, then its Fleron polyomino has perimeter $4c$.
\end{lemma}

\begin{proof}
  This is clear by definition.  Indeed, every link in a Tangle corresponds to exactly one exterior edge of its Fleron polyomino.
\end{proof}

\section{Growth constants}

Suppose $a_p(m)$ is the number of fixed hole-free polyominoes with $m$ cells and $\ell_p(c)$ is the number of fixed hole-free polyominoes with perimeter $2c$.

It is well-known that the limits
\begin{equation*}
  \kappa_p = \lim_{m\to\infty}a_p(m)^{1/m}\text{ and }
  \mu_p = \lim_{c\to\infty}\ell_p(c)^{1/c}
\end{equation*}
exist.  These are known as \textit{growth constants} or \textit{connective constants}, and it is estimated that $\kappa_p\approx  3.97094397$ \cite{gjwe} and $\mu_p\approx 2.63815853035$ \cite{clisbyjensen}.

In this section, we will investigate the existence of corresponding growth constants for Tangles.

For each positive integer $m$, we define $a_0(m)$, $a_1(m)$, and $a_2(m)$ to be the number of distinct fixed, one-sided, and free Tangles, respectively, with size $m$, or equivalently by Theorem \ref{tangle area}, area $(4m+\pi)r^2$.

Similarly, for each positive integer $c$, we define $\ell_0(c)$, $\ell_1(c)$, and $\ell_2(c)$ to be the number of distinct fixed, one-sided, and free Tangles, respectively, with class $c$, or equivalently, length $4c$.

\begin{lemma}\label{liminf and limsup bounds}
  The following inequalities are true.
  \begin{align*}
    \kappa_p\leq\liminf_{m\to\infty}a_0(m)^{1/m}\quad &
    \limsup_{m\to\infty}a_0(m)^{1/m}\leq\kappa_p^2\\
    \mu_p\leq\liminf_{c\to\infty}\ell_0(c)^{1/c}\quad &
    \limsup_{c\to\infty}\ell_0(c)^{1/c}\leq\mu_p^2\quad
  \end{align*}
\end{lemma}

\begin{proof}
  By Lemmas \ref{chan area} and \ref{fleron area},
  \begin{gather*}
    a_p(m+1)\leq a_0(m)\leq a_p(2m+1)\\
    \left(a_p(m+1)^{1/(m+1)}\right)^{(m+1)/m}\leq a_0(m)^{1/m}\leq\left(a_p(2m+1)^{1/(2m+1)}\right)^{(2m+1)/m},
  \end{gather*}
  and the first two results follow by taking limits.  The last two results are similar, but use Lemmas \ref{chan perimeter} and \ref{fleron perimeter}.
\end{proof}

We now adapt an argument by Klarner for polyominoes \cite{klarner} to the Tangle case and show that the limits
\begin{equation*}
  \kappa = \lim_{m\to\infty}a_0(m)^{1/m}\text{ and }
  \mu = \lim_{c\to\infty}\ell_0(c)^{1/c}
\end{equation*}
exist.

In fact, since the dihedral group of symmetries of the square has eight elements, we have
\begin{equation*}
  \frac{1}{8}a_0(m)\leq a_2(m) \leq a_1(m) \leq a_0(m)
\end{equation*}
\begin{equation*}
  \frac{1}{8}\ell_0(c)\leq \ell_2(c) \leq \ell_1(c) \leq \ell_0(c),
\end{equation*}
and since $\lim_{n\to\infty}\left(\frac{1}{8}\right)^{1/n}=1$, it follows by the sandwich theorem that
\begin{equation*}
  \kappa = \lim_{m\to\infty}a_i(m)^{1/m}\text{ and }
  \mu = \lim_{c\to\infty}\ell_i(c)^{1/c}
\end{equation*}
for all $i$, i.e., the growth constants are the same for fixed, one-sided, and free Tangles.

\begin{lemma}\label{log-superadditive}
  The functions $a_0$ and $\ell_0$ are log-superadditive, i.e.,
  \begin{align*}
    a_0(m_1)a_0(m_2) &\leq a_0(m_1+m_2)\\
    \ell_0(c_1)\ell_0(c_2) &\leq \ell_0(c_1+c_2)
  \end{align*}
\end{lemma}

\begin{proof}
  First consider Tangles enumerated by area.  Take one fixed Tangle of size $m_1$ and another fixed Tangle of size $m_2$, and then translate their dual graphs so that the rightmost vertex in the top row of the first dual graph coincides with the leftmost vertex in the bottom row of the second dual graph.  This one-to-one operation results in the dual graph of a Tangle with size $m_1+m_2$.

  Now consider Tangles enumerated by length.  Take one fixed Tangle of length $4c_1$ and another fixed Tangle of length $4c_2$, and then translate their dual graphs so that the rightmost vertex in the top row of the first dual graph is directly below the leftmost vertex in the bottom row of the second dual graph.  Now connect these two vertices with an edge.  This one-to-one operation results in the dual graph of a Tangle with length $4(c_1+c_2)$.  Indeed, one link from each of the original Tangles was removed in the process, but two new ones were added surrounding the new edge.  See Figure \ref{log-superadditive proof}.
\end{proof}

\begin{figure}
  \begin{tikzpicture}
    \path (0,0) node[draw,shape=circle,fill=gray] (a) {};
    \path (0,{sqrt(2)}) node[draw,shape=circle,fill=gray] (b) {};
    \draw[line width=4pt,lightgray] (45:0.5) arc (45:135:0.5);
    \draw[line width=4pt,lightgray] (b) + (-45:0.5) arc (-45:-135:0.5);
    \draw[line width=4pt] (45:0.5) arc (225:135:0.5);
    \draw[line width=4pt] (135:0.5) arc (-45:45:0.5);
    \draw[dashed] (a) -- (b);
    \foreach \i in {0, 1, 2, 3}
    \fill[white] (0, {1/sqrt(2)}) + ({45 + 90 * \i}:0.5) circle (1.5pt);
  \end{tikzpicture}
  \caption{Proof of Lemma \ref{log-superadditive}}
  \label{log-superadditive proof}
\end{figure}

\begin{theorem}
  The Tangle growth constants $\kappa$ and $\mu$ defined above exist.  Furthermore,
  \begin{gather*}
    \kappa_p\leq \kappa \leq \kappa_p^2\\
    \mu_p \leq \mu \leq \mu_p^2.
  \end{gather*}
\end{theorem}

\begin{proof}
  We prove the existence of $\kappa$.  The proof for $\mu$ is essentially the same.

  Taking logarithms and then negating the inequality from Lemma \ref{log-superadditive}, we have
  \begin{equation*}
    -\log a_0(m_1+m_2)\leq-\log a_0(m_1)-\log a_0(m_2),
  \end{equation*}
  and so by Fekete's subadditive lemma \cite{fekete},
  \begin{equation*}
    \lim_{m\to\infty}\frac{\log a_0(m)}{m} = \lim_{m\to\infty}\log a_0(m)^{1/m}
  \end{equation*}
  exists so long as $(a_0(m)^{1/m})$ is bounded.  But it is by Lemma \ref{liminf and limsup bounds}.

  The bounds on $\kappa$ and $\mu$ follow directly from Lemma \ref{liminf and limsup bounds}.
\end{proof}

\section*{Acknowledgements}

Many thanks go to the author's son, Gabriel, whose love of toy trains inspired this paper.  Thanks also to Ron Taylor, who introduced the author to existing literature on Tangles, to Julian Fleron for providing a copy of \cite{fleron1}, and to the anonymous referee for several useful comments.

\end{document}